\documentclass[a4paper,10pt,notitlepage]{article}
\usepackage[margin=3cm]{geometry}
\usepackage{amsfonts}
\usepackage{amsmath}
\usepackage{amssymb}
\usepackage{amsthm}
\usepackage{float}
\newtheorem{thm}{Theorem}
\newtheorem{lemma}[thm]{Lemma}
\newtheorem{prop}[thm]{Proposition}
\newtheorem{question}[thm]{Question}
\newtheorem{conj}[thm]{Conjecture}

\begin{document}

\title{Tiling with punctured intervals}
\author{Harry Metrebian}
\date{}
\maketitle

\begin{abstract}
	\noindent
	It was shown by Gruslys, Leader and Tan that any finite subset of $\mathbb{Z}^n$ tiles $\mathbb{Z}^d$ for some $d$. The first non-trivial case is the punctured interval, which consists of the interval $\{-k,\ldots,k\} \subset \mathbb{Z}$ with its middle point removed: they showed that this tiles $\mathbb{Z}^d$ for $d = 2k^2$, and they asked if the dimension needed tends to infinity with $k$. In this note we answer this question: we show that, perhaps surprisingly, every punctured interval tiles $\mathbb{Z}^4$.
\end{abstract}

\section{Introduction}

A \emph{tile} is a finite non-empty subset of $\mathbb{Z}^n$ for some $n$. We say that a tile $T$ \emph{tiles} $\mathbb{Z}^d$ if $\mathbb{Z}^d$ can be partitioned into copies of $T$, that is, subsets that are translations, rotations or reflections, or any combination of these, of $T$.

For example, the tile $\texttt{X.X} = \{-1,1\} \subset \mathbb{Z}$ tiles $\mathbb{Z}$. The tile $\texttt{XX.XX} = \{-2,-1,1,2\} \subset \mathbb{Z}$ does not tile $\mathbb{Z}$, but we can also regard it as a tile in $\mathbb{Z}^2$, and indeed it tiles $\mathbb{Z}^2$, as shown, for example, in \cite{gltan16}.

Chalcraft \cite{chalcraft1,chalcraft2} conjectured that, for any tile $T \subset \mathbb{Z}^n$, there is some dimension $d$ for which $T$ tiles $\mathbb{Z}^d$. This was proved by Gruslys, Leader and Tan \cite{gltan16}. The first non-trivial case is the \emph{punctured interval} $T = \underbrace{\texttt{XXXXX}}_{k}\!\texttt{.}\!\underbrace{\texttt{XXXXX}}_{k}$. The authors of \cite{gltan16} showed that $T$ tiles $\mathbb{Z}^d$ for $d = 2k^2$, but they were unable to prove that the smallest required dimension $d$ was quadratic in $k$, or even that $d \to \infty$ as $k \to \infty$. They therefore asked the following question:

\begin{question}[Gruslys, Leader, Tan \cite{gltan16}]
	Let $T$ be the punctured interval $\underbrace{\texttt{\emph{XXXXX}}}_{k}\!\texttt{.}\!\underbrace{\texttt{\emph{XXXXX}}}_{k}$, and let $d$ be the least number such that $T$ tiles $\mathbb{Z}^d$. Does $d \to \infty$ as $k \to \infty$?
\end{question}

In this paper we will show that, rather unexpectedly, $d$ does not tend to $\infty$:

\begin{thm}\label{mainthm}
	Let $T$ be the punctured interval $\underbrace{\texttt{\emph{XXXXX}}}_{k}\!\texttt{.}\!\underbrace{\texttt{\emph{XXXXX}}}_{k}$. Then $T$ tiles $\mathbb{Z}^4$. Furthermore, if $k$ is odd or congruent to $4 \pmod 8$, then $T$ tiles $\mathbb{Z}^3$.
\end{thm}

We have already noted that \texttt{X.X} tiles $\mathbb{Z}$, and \texttt{XX.XX} tiles $\mathbb{Z}^2$ but not $\mathbb{Z}$. It can be shown via case analysis that, for $k \geq 3$, the tile $T$ does not tile $\mathbb{Z}^2$. However, this proof is tedious and provides little insight, and since it is not the focus of this paper, we omit it. For odd $k \geq 3$ and for $k \equiv 4 \pmod 8$, 3 is therefore the least $d$ such that $T$ tiles $\mathbb{Z}^d$. For the remaining cases, namely $k \equiv 0, 2, 6 \pmod 8$, $k \geq 6$, it is unknown whether the least possible $d$ is 3 or 4.

In this paper, we will first prove the result for odd $k$. This will introduce some key ideas, which we will develop to prove the result for general $k$, and then to improve the dimension from 4 to 3 for $k \equiv 4 \pmod 8$.

Finally, we give some background. Tilings of $\mathbb{Z}^2$ by polyominoes (edge-connected tiles in $\mathbb{Z}^2$) have been thoroughly investigated. For example, Golomb \cite{golomb70} showed that results of Berger \cite{berger66} implied that there is no algorithm which decides whether copies of a given finite set of polyominoes tile $\mathbb{Z}^2$. It is unknown whether the same is true for tilings by a single polyomino. For tilings of $\mathbb{Z}$ by sets of general one-dimensional tiles, such an algorithm does exist, as demonstrated by Adler and Holroyd \cite{ah81}. Kisisel \cite{kisisel01} introduced an ingenious technique for proving that certain tiles do not tile $\mathbb{Z}^2$ without having to resort to case analysis.

A similar problem is to consider whether a tile $T$ tiles certain finite regions, such as cuboids. There is a significant body of research, sometimes involving computer searches, on tilings of rectangles in $\mathbb{Z}^2$ by polyominoes (see, for example, Conway and Lagarias \cite{cl90} and Dahlke \cite{dahlke}). Friedman \cite{friedman} has collected some results on tilings of rectangles by small one-dimensional tiles. More recently, Gruslys, Leader and Tomon \cite{gltomon16} and Tomon \cite{tomon16} considered the related problem of partitioning the Boolean lattice into copies of a poset, and similarly Gruslys \cite{gruslys16} and Gruslys and Letzter \cite{gl16} have worked on the problem of partitioning the hypercube into copies of a graph.

\section{Preliminaries and the odd case}

We begin with the case of $k$ odd. This is technically much simpler than the general case, and allows us to demonstrate some of the main ideas in the proof of Theorem \ref{mainthm} in a less complicated setting.

\begin{thm}\label{kodd}
	Let $T$ be the punctured interval $\underbrace{\texttt{\emph{XXXXX}}}_{k}\!\texttt{.}\!\underbrace{\texttt{\emph{XXXXX}}}_{k}$, with $k$ odd. Then $T$ tiles $\mathbb{Z}^3$.
\end{thm}

Throughout this section, $T$ is fixed, and $k \geq 3$. We will not yet assume that $k$ is odd, because the tools that we are about to develop will be relevant to the general case too.

We start with an important definition from \cite{gltan16}: a \emph{string} is a one-dimensional infinite line in $\mathbb{Z}^d$ with every $(k+1)$th point removed. Crucially, a string is a disjoint union of copies of $T$.

We cannot tile $\mathbb{Z}^d$ with strings, as each string intersects $[k+1]^d$ in either 0 or $k$ points, and $(k+1)^d$ is not divisible by $k$. However, we could try to tile $\mathbb{Z}^d$ by using strings in $d-1$ of the $d$ possible directions, leaving holes that can be filled with copies of $T$ in the final direction. We therefore consider $\mathbb{Z}^d$ as consisting of slices equivalent to $\mathbb{Z}^{d-1}$, each of which will be partially tiled by strings.

Any partial tiling of the discrete torus $\mathbb{Z}_{k+1}^{d-1} = (\mathbb{Z}/(k+1)\mathbb{Z})^{d-1}$ by lines with one point removed corresponds to a partial tiling of $\mathbb{Z}^{d-1}$ by strings. We will restrict our attention to these tilings at first, as they are easy to work with.

We will call a set $X \subset \mathbb{Z}_{k+1}^{d-1}$ a \emph{hole} in $\mathbb{Z}_{k+1}^{d-1}$ if $\mathbb{Z}_{k+1}^{d-1} \setminus X$ can be tiled with strings. One particularly useful case of this is when $d = 3$ and $X$ either has exactly one point in each row of $\mathbb{Z}_{k+1}^2$ or exactly one point in each column of $\mathbb{Z}_{k+1}^2$. Then $X$ is clearly a hole, since a string in $\mathbb{Z}_{k+1}^2$ is just a row or column minus a point.

The following result will allow us to fill the gaps in the final direction, assuming we have chosen the partial tilings of the $\mathbb{Z}^{d-1}$ slices carefully:

\begin{lemma}\label{biglemma}
	Let $S \subset \mathbb{Z}^d$, $|S| = 3$. Then there exists $Y \subset S \times \mathbb{Z}$ such that $T$ tiles $Y$, and for every $n \in \mathbb{Z}$, $|Y \cap (S \times \{n\})| = 2$.
\end{lemma}

\begin{proof}
	Let $S = \{x_1, x_2, x_3\}$. For $i = 1,2,3$, place a copy of $T$ beginning at $\{x_i\} \times \{n\}$ for every $n \equiv ik \pmod {3k}$. The union $Y$ of these tiles has the required property:\\
	For $n \equiv 0, k+1, \ldots, 2k-1 \pmod{3k}$, $Y \cap (S \times \{n\}) = \{x_1, x_3\} \times \{n\}$.\\
	For $n \equiv k, 2k+1, \ldots, 3k-1 \pmod{3k}$, $Y \cap (S \times \{n\}) = \{x_1, x_2\} \times \{n\}$.\\
	For $n \equiv 2k, 1, \ldots, k-1 \pmod{3k}$, $Y \cap (S \times \{n\}) = \{x_2, x_3\} \times \{n\}$.\\
\end{proof}

We will now prove Theorem \ref{kodd}. We know that if $X \subset \mathbb{Z}_{k+1}^2$ has one point in each row or column then $X$ is a hole of size $k+1$. Since $k+1$ is even, we can try to choose $X_n$ in each slice $\mathbb{Z}_{k+1}^2 \times \{n\}$ so that $\bigcup_{n\in\mathbb{Z}}X_n$ is the disjoint union of $\frac{k+1}{2}$ sets $Y_i$ of the form in Lemma \ref{biglemma}.

We can do this as follows:\\
For $n \equiv 0, k+1, \ldots, 2k-1 \pmod{3k}$, let $X_n = \{(0,0),(1,1),\ldots,(k-1,k-1),(k,k)\}$.\\
For $n \equiv k, 2k+1, \ldots, 3k-1 \pmod{3k}$, let $X_n = \{(0,0),(0,1),(2,2),(2,3),\ldots,(k-1,k-1),\newline(k-1,k)\}$.\\
For $n \equiv 2k, 1, \ldots, k-1 \pmod{3k}$, let $X_n = \{(0,1),(1,1),(2,3),(3,3),\ldots,(k-1,k),(k,k)\}$.\\
Then let $X = \bigcup\limits_{n\in\mathbb{Z}} (X_n \times \{n\}) \subset \mathbb{Z}_{k+1}^2 \times \mathbb{Z}$.

Each $X_n$ is a hole, so we can tile $(\mathbb{Z}_{k+1}^2 \times \mathbb{Z})\setminus X$ with strings. Also, $X$ is the disjoint union of sets of the form $Y$ from Lemma \ref{biglemma}: for $0 \leq i \leq \frac{k-1}{2}$, let $S_i = \{(2i,2i),(2i,2i+1),(2i+1,2i+1)\}$. Then $X \cap (S_i \times \mathbb{Z})$ is precisely the set $Y$ generated from $S_i$ in the proof of Lemma \ref{biglemma}. Hence $T$ tiles $X$.

Since $(\mathbb{Z}_{k+1}^2 \times \mathbb{Z})\setminus X$ can be tiled with strings, we can partially tile $\mathbb{Z}^3$ with strings, leaving a copy of $X$ empty in each copy of $\mathbb{Z}_{k+1}^2 \times \mathbb{Z}$. We can tile all of these copies of $X$ with $T$, so $T$ tiles $\mathbb{Z}^3$, completing the proof of Theorem \ref{kodd}.

\section{The general case}

We now move on to general $k$:

\begin{thm}\label{generalk}
	Let $T$ be the tile $\underbrace{\texttt{\emph{XXXXX}}}_{k}\!\texttt{.}\!\underbrace{\texttt{\emph{XXXXX}}}_{k}$. Then $T$ tiles $\mathbb{Z}^4$.
\end{thm}

We will assume throughout that $T$ is fixed and $k \geq 3$.

For even $k$, the construction used to prove Theorem \ref{kodd} does not work, as all holes in $\mathbb{Z}_{k+1}^2$ have size $(k+1)^2-mk$ for some $m$, and this is always odd, so we cannot use Lemma \ref{biglemma}. The same is true if we replace 2 with a larger dimension, or if, as in \cite{gltan16}, we use strings in which every $(2k+1)$th point, rather than every $(k+1)$th point, is removed. We will therefore need a new idea.

Instead of using strings in $d-1$ out of $d$ directions, we could only use them in $d-2$ directions and fill the gaps with copies of $T$ in the 2 remaining directions. We will show that this approach works in the case $d = 2$, giving a tiling of $\mathbb{Z}^4$. The strategy will be to produce a partial tiling of each $\mathbb{Z}^3$ slice and use the construction from Lemma \ref{biglemma} to fill the gaps with tiles in the fourth direction.

We will again build partial tilings of $\mathbb{Z}^{2}$, and therefore of higher dimensions, from partial tilings of the discrete torus $\mathbb{Z}_{k+1}^{2}$. The following result is a special case of one proved in \cite{gltan16}:

\begin{prop}\label{onepoint}
	If $x \in \mathbb{Z}_{k+1}^{2}$, then $\mathbb{Z}_{k+1}^{2}\setminus\{x\}$ can be tiled with strings.
\end{prop}

\begin{proof}
	Let $x = (x_1,x_2)$, where the first coordinate is horizontal and the second vertical. Since a string is a row or column minus one point, we can place a string $(\{n\} \times \mathbb{Z}_{k+1})\setminus\{(n,x_2)\}$ in each column, leaving only the row $\mathbb{Z}_{k+1} \times \{x_2\}$ empty. Placing the string $(\mathbb{Z}_{k+1} \times \{x_2\})\setminus \{x\}$ in this row completes the tiling of $\mathbb{Z}_{k+1}^{2}\setminus\{x\}$.
\end{proof}

The sets $S$ of size 3 that we will use in Lemma \ref{biglemma} will have 2 points, say $x_1$ and $x_2$, in one $\mathbb{Z}_{k+1}^{2}$ layer and one point, say $x_3$, in another layer. Every layer will contain points from exactly one such set $S$. Let $Y$ be the set constructed from $S$ in the proof of Lemma \ref{biglemma}. In a given slice $\mathbb{Z}^3 \times \{n\}$, there are therefore two cases:
\begin{enumerate}
	\item $Y \cap (S \times \{n\}) = \{x_1, x_3\} \times \{n\}$ or $\{x_2, x_3\} \times \{n\}$.
	\item $Y \cap (S \times \{n\}) = \{x_1, x_2\} \times \{n\}$.
\end{enumerate}

In Case 1, each $\mathbb{Z}_{k+1}^{2}$ layer contains exactly one point of $Y$. $T$ then tiles the rest of the layer by Proposition \ref{onepoint}.

In Case 2, some of the layers contain two points of $Y$, and some of the layers contain no points. Holes of size 0 and 2 do not exist, so we will need copies of $T$ in the third direction to fill some gaps (where $Y$ consists of copies of $T$ in the fourth direction). The following lemma provides us with a way to do this:

\begin{lemma}\label{otherlemma}
	Let $A \subset \mathbb{Z}^d$, $|S| = 3k$. Then there exists $B \subset S \times \mathbb{Z}$ such that $T$ tiles $B$, and
	\[|B \cap (S \times \{n\})| = 
	\begin{cases}
	k+1 & \text{\emph{if} } n \equiv 1, \ldots, k \pmod{2k}\\
	k-1 & \text{\emph{if} } n \equiv k+1, \ldots, 2k \pmod{2k}
	\end{cases}\]
\end{lemma}

\begin{proof}
	Let $A = \{a_1, \ldots, a_{3k}\}$. Then:\\
	For $i = 1, \ldots, k$, place a copy of $T$ beginning at $\{a_i\} \times \{n\}$ for every $n \equiv i \pmod{6k}$.\\
	For $i = k+1, \ldots, 2k$, place a copy of $T$ beginning at $\{a_i\} \times \{n\}$ for every $n \equiv i+k \pmod{6k}$.\\
	For $i = 2k+1, \ldots, 3k$, place a copy of $T$ beginning at $\{a_i\} \times \{n\}$ for every $n \equiv i+2k \pmod{6k}$.\\
	We now observe that the union $B$ of these tiles has the required property.\\
	For $n \equiv 1, \ldots, k \pmod{6k}$, $B \cap (A \times \{n\}) = \{a_{2k+n}, \ldots, a_{3k}, a_1, \ldots, a_n\}$ (size $k+1$).\\
	For $n \equiv k+1, \ldots, 2k \pmod{6k}$, $B \cap (A \times \{n\}) = \{a_1, \ldots, a_k\}\setminus\{a_{n-k}\}$ (size $k-1$).\\
	For $n \equiv 2k+1, \ldots, 3k \pmod{6k}$, $B \cap (A \times \{n\}) = \{a_{n-2k}, \ldots, a_{n-k}\}$ (size $k+1$).\\
	For $n \equiv 3k+1, \ldots, 4k \pmod{6k}$, $B \cap (A \times \{n\}) = \{a_{k+1}, \ldots, a_{2k}\}\setminus\{a_{n-2k}\}$ (size $k-1$).\\
	For $n \equiv 4k+1, \ldots, 5k \pmod{6k}$, $B \cap (A \times \{n\}) = \{a_{n-3k}, \ldots, a_{n-2k}\}$ (size $k+1$).\\
	For $n \equiv 5k+1, \ldots, 6k \pmod{6k}$, $B \cap (A \times \{n\}) = \{a_{2k+1}, \ldots, a_{3k}\}\setminus\{a_{n-3k}\}$ (size $k-1$).
\end{proof}

The reasoning behind this lemma is that there exist sets $X \subset \mathbb{Z}_{k+1}^{2} \times \mathbb{Z}$ that are missing exactly $k+1$ points in every $\mathbb{Z}_{k+1}^{2}$ layer and can be tiled with strings. If we take $d = 2$ in Lemma \ref{otherlemma}, we would like to choose such a set $X$ and a set $A \subset \mathbb{Z}_{k+1}^{2}$ (abusing notation slightly, as $\mathbb{Z}_{k+1}^{2}$ is not actually a subset of $\mathbb{Z}^2$) such that the resulting $B$ in Lemma \ref{otherlemma} is disjoint from $X$. Then $(\mathbb{Z}_{k+1}^{2} \times \mathbb{Z})\setminus(B \cup X)$ contains either 2 or 0 points in each $\mathbb{Z}_{k+1}^{2}$ layer, which is what we wanted.

In order for this construction to work, we need the set $B \cap (A \times \{n\})$ to be a hole whenever it has size $k+1$, and to be a subset of a hole of size $k+1$ whenever it has size $k-1$, so that we actually can tile the required points with strings. By observing the forms of the sets $B \cap (A \times \{n\})$ in the proof of Lemma \ref{otherlemma}, we see that it is sufficient to choose the $a_n$ such that for all $n$, $\{a_n, \ldots, a_{n+k}\}$ is a hole. Here we regard the indices $n$ of the points $a_n$ of $A$ as integers mod $3k$, so $a_{3k+1} = a_1$ and so on. The following proposition says that we can do this.

\begin{prop}\label{anprop}
	There exists a set $A = \{a_1, \ldots, a_{3k}\} \subset \mathbb{Z}_{k+1}^{2}$ such that for all $n$, $\{a_n, \ldots, a_{n+k}\}$ contains either one point in every row or one point in every column. Here the indices are regarded as integers \emph{mod} $3k$.
\end{prop}

\begin{proof}
	For $n = 1, \ldots, k+1$, let $a_n = (n-1,n-1)$.\\
	For $n = k+2, \ldots, 2k-1$, let $a_n = (n-k-2,n-k-1)$.\\
	For $n = 2k, 2k+1, 2k+2$, let $a_n = (n-k-2,n-2k)$.\\
	For $n = 2k+3, \ldots, 3k$, let $a_n = (n-2k-3,n-2k)$.\\
	Note that all the $a_n$ are distinct. Let us regard the first coordinate as horizontal and the second as vertical.\\
	Then, for $n = 1, \ldots, 2k$, $\{a_n, \ldots, a_{n+k}\}$ contains one point in every column.\\
	For $n = 2k+1, \ldots, 3k$, $\{a_n, \ldots, a_{n+k}\}$ contains one point in every row.
\end{proof}

From now on, $a_n$ refers to the points defined in the above proof. This proposition is the motivation for choosing the value $6k$ in the proof of Lemma \ref{otherlemma}.

We can now prove Theorem \ref{generalk}. We will need 3 distinct partial tilings of $\mathbb{Z}^3$ slices, corresponding to the 3 cases in the proof of Lemma \ref{biglemma} with $d = 3$. The repeating unit in each of these partial tilings will have size $(k+1) \times (k+1) \times 6k$, so we will work in $\mathbb{Z}_{k+1}^2 \times \mathbb{Z}_{6k}$.

We start by choosing the sets $S$ as in Lemma \ref{biglemma}. These will be as follows:\\
For $n = 1, \ldots, k$, $S_n = \{(0,0,n),(a_n,n+k),(a_{k+1},n+k)\}$.\\
For $n = k+1, \ldots, 2k$, $S_n = \{(0,0,n+k),(a_n,n+2k),(a_{2k+1},n+2k)\}$.\\
For $n = 2k+1, \ldots, 3k$, $S_n = \{(0,0,n+2k),(a_n,n+3k),(a_1,n+3k)\}$.\\
We will refer to the points in $S_n$ as $x_{n,1},x_{n,2},x_{n,3}$ in the order given.

We can construct a set $Y_n \subset \mathbb{Z}^4$ from each $S_n$ using the construction in the proof of Lemma \ref{biglemma}. Let $Y = \bigcup_{1 \leq n \leq 3k} Y_n$. For a given $m \in \mathbb{Z}$, there are two possibilities for the structure of $Y \cap (\mathbb{Z}_{k+1}^2 \times \mathbb{Z}_{6k} \times \{m\})$:
\begin{enumerate}
	\item $Y \cap (\mathbb{Z}_{k+1}^2 \times \mathbb{Z}_{6k} \times \{m\})$ consists of pairs of the form $\{x_{n,1},x_{n,2}\}$ or $\{x_{n,1},x_{n,3}\}$. Then it contains exactly one point in each $\mathbb{Z}_{k+1}^2$ layer. We can therefore tile $(\mathbb{Z}_{k+1}^2 \times \mathbb{Z}_{6k} \times \{m\}) \setminus Y$ entirely with strings, by Proposition \ref{onepoint}.
	\item $Y \cap (\mathbb{Z}_{k+1}^2 \times \mathbb{Z}_{6k} \times \{m\})$ consists of pairs of the form $\{x_{n,2},x_{n,3}\}$. Then it contains either 2 or 0 points in each $\mathbb{Z}_{k+1}^2$ layer.\\
	If $A = \{a_1, \ldots, a_{3k}\}$, and $B$ is the set constructed from $A$ in the proof of Lemma \ref{otherlemma}, then, by the choice of the $S_n$, the sets $B$ and $Y \cap (\mathbb{Z}_{k+1}^2 \times \mathbb{Z}_{6k} \times \{m\})$ are disjoint. Furthermore, if $C$ is the union of these two sets, then, for every $n$, $C \cap (\mathbb{Z}_{k+1}^2 \times \{n\} \times \{m\}) = \{a_r, \ldots, a_{r+k}\}$ for some $r$, and by Proposition \ref{anprop}, this contains either one point in every row or one point in every column and is therefore a hole.\\
	Since $T$ tiles $B$, it also tiles $(\mathbb{Z}_{k+1}^2 \times \mathbb{Z}_{6k} \times \{m\}) \setminus Y$.
\end{enumerate}

$T$ tiles $Y$ by Lemma \ref{biglemma}. Hence $T$ tiles $\mathbb{Z}_{k+1}^2 \times \mathbb{Z}_{6k} \times \mathbb{Z}$, and therefore also $\mathbb{Z}^4$, completing the proof of Theorem \ref{generalk}.

\section{The 4 mod 8 case}

To finish the proof of Theorem \ref{mainthm}, all that remains is to prove the following:

\begin{thm}\label{4mod8}
	Let $T$ be the tile $\underbrace{\texttt{\emph{XXXXX}}}_{k}\!\texttt{.}\!\underbrace{\texttt{\emph{XXXXX}}}_{k}$, with $k \equiv 4 \pmod 8$. Then $T$ tiles $\mathbb{Z}^3$.
\end{thm}

We will prove this by constructing partial tilings of each $\mathbb{Z}^2$ slice and filling in the gaps using the construction from the proof of Lemma \ref{biglemma}. We will define 3 subsets $X_1$, $X_2$, $X_3$ of $\mathbb{Z}^2$ and show that $T$ tiles each of them. However, two of these tilings will not make use of strings.

Let $S_1 = \{(x,x+n(k+1)) \; | \; n \in \mathbb{Z}, x \equiv 2n,2n+1,2n+2,2n+3 \pmod 8\}$.

Let $S_2 = \{(x,x+n(k+1)) \; | \; n\in \mathbb{Z}, x \equiv 2n+4,2n+5,2n+6,2n+7 \pmod 8\}$.

Let $S_3 = \{(x,x+n(k+1)+1) \; | \; n \in \mathbb{Z}, x \equiv 2n+2,2n+3,2n+4,2n+5 \pmod 8\}$.

Let $X_1 = \mathbb{Z}^2 \setminus (S_2 \cup S_3)$, $X_2 = \mathbb{Z}^2 \setminus (S_1 \cup S_3)$, $X_3 = \mathbb{Z}^2 \setminus (S_1 \cup S_2)$.

Let the first coordinate be horizontal and the second vertical.

$X_3$ is $\mathbb{Z}^2$ with every $(k+1)$th diagonal removed, so each row (or column) is $Z$ with every $(k+1)$th point removed, that is, a string. Hence $T$ tiles $X_3$.

We will show that $X_1$ can be tiled with vertical copies of $T$ and $X_2$ can be tiled with horizontal copies of $T$.

Note that $(x,x+n(k+1))+(2,k+3) = (x+2,(x+2)+(n+1)(k+1))$. Also, if $x \equiv 2n+r \pmod 8$, then $x+2 \equiv 2(n+1)+r \pmod 8$. Hence, by the definitions of $S_2$ and $S_3$, we see that $X_1$ is invariant under translation by $(2,k+3)$. To show that vertical copies of $T$ tile $X_1$, it therefore suffices to show that $T$ tiles the columns $X_1 \cap (\{0\} \times \mathbb{Z})$ and $X_1 \cap (\{1\} \times \mathbb{Z})$.

But in fact, if $(0,y) \in S_2$, then $0 \equiv 2n+4$ or $2n+6 \pmod 8$, so $1 \equiv 2n+5$ or $2n+7 \pmod 8$, so also $(1,y+1) \in S_2$. The converse also holds, and the same is true for $S_3$. Thus we only need to check the case $x = 0$.

$(0,n(k+1)) \in S_2$ for $n \equiv 1,2,5,6 \pmod 8$, that is, $n \equiv 1,2 \pmod 4$.

$(0,n(k+1)+1) \in S_3$ for $n \equiv 2,3,6,7 \pmod 8$, that is, $n \equiv 2,3 \pmod 4$.

Therefore $(0,y) \notin X_1$ for $y \equiv k+1, 2(k+1), 2(k+1)+1, 3(k+1)+1 \pmod{4(k+1)}$, so copies of $T$ beginning at positions $1$ and $2(k+1)+2 \pmod{4(k+1)}$ tile $X_1 \cap (\{0\} \times \mathbb{Z})$.

Hence $T$ tiles $X_1$.

Note that $(x,x+n(k+1))+(k+2,1) = (x+k+2,(x+k+2)+(n-1)(k+1))$.\\
Since $k \equiv 4 \pmod 8$, if $x \equiv 2n+r \pmod 8$ then $x+k+2 \equiv 2(n-1)+r \pmod 8$. Hence $X_2$ is invariant under translation by $(k+2,1)$, by the definitions of $S_1$ and $S_3$. To show that horizontal copies of $T$ tile $X_2$, it is therefore enough to show that $T$ tiles the row $X_2 \cap (\mathbb{Z} \times \{0\})$.

We can express $S_1$ as $\{(y-n(k+1),y) \; | \; y \equiv -n,1-n,2-n,3-n \pmod 8\}$.

Similarly $S_3 = \{(y-n(k+1)-1,y) \; | \; y \equiv 3-n,4-n,5-n,6-n \pmod 8\}$.

Therefore $(-n(k+1),0) \in S_1$ for $n \equiv 0,1,2,3 \pmod 8$, and $(-n(k+1)-1,0) \in S_3$ for $n \equiv 3,4,5,6 \pmod 8$.

Hence $(x,0) \notin X_2$ for $x \equiv 0, 2(k+1)-1, 3(k+1)-1, 4(k+1)-1, 5(k+1)-1, 5(k+1), 6(k+1), \newline 7(k+1) \pmod{8(k+1)}$, so copies of $T$ beginning at positions $k+1, 3(k+1), 5(k+1)+1, 7(k+1)+1 \pmod{8(k+1)}$ tile $X_2 \cap (\mathbb{Z} \times \{0\})$.

Hence $T$ tiles $X_2$.

$S_1 \cup S_2 \cup S_3$ can be partitioned into sets of the form $S = \{x_1, x_2, x_3\}$, where $x_1 = (x,y) \in S_1$, $x_2 = (x+4,y+4) \in S_2$, $x_3 = (x+2,y+3) \in S_3$. Then $|S| = 3$, so we can construct the corresponding set $Y \subset \mathbb{Z}^3$ as in Lemma \ref{biglemma}. Now, given $n \in \mathbb{Z}$, $(S \times \{n\}) \setminus Y = \{x_i\}$ for some $i \in \{1,2,3\}$. Then $Y \cap (X_i \times \{n\}) = \emptyset$. If we do this for all such sets $S$, and let $U$ be the (disjoint) union of the resulting sets $Y$, then $U \cap (X_i \times \{n\}) = \emptyset$, and $\mathbb{Z}^2 \times \{n\} \subset U \cup (X_i \times \{n\})$. Recall that $T$ tiles each $Y$ and therefore $U$.

We can do this for every $n$, choosing a partial tiling $X_i$ for the corresponding $\mathbb{Z}^2$ layer. Together with $U$, these form a tiling of $\mathbb{Z}^3$ by $T$. This completes the proof of Theorem \ref{4mod8}, and therefore also the proof of Theorem \ref{mainthm}.

\section{Open problems}

Theorem \ref{mainthm}, together with the result that a punctured interval $T = \underbrace{\texttt{XXXXX}}_{k}\!\texttt{.}\!\underbrace{\texttt{XXXXX}}_{k}$ does not tile $\mathbb{Z}^2$ for $k \geq 3$, determines the smallest dimension $d$ such that $T$ tiles $\mathbb{Z}^d$ in the cases $k$ odd and $k \equiv 4 \pmod 8$. However, for other values of $k$, it is still unknown whether the smallest such dimension $d$ is 3 or 4:

\begin{question}
	Let $T$ be the punctured interval $\underbrace{\texttt{\emph{XXXXX}}}_{k}\!\texttt{.}\!\underbrace{\texttt{\emph{XXXXX}}}_{k}$, where $k \equiv 0, 2, 6 \pmod 8$, $k \geq 6$. Does $T$ tile $\mathbb{Z}^3$?
\end{question}

It is also natural to consider more general tiles. The next non-trivial case is that of an interval with a non-central point removed. One might wonder if there is an analogue of Theorem \ref{mainthm} for these tiles:

\begin{question}
	Does there exist a number $d$ such that, for any tile $T$ consisting of an interval in $\mathbb{Z}$ with one point removed, $T$ tiles $\mathbb{Z}^d$?
\end{question}

For general one-dimensional tiles, Gruslys, Leader and Tan \cite{gltan16} conjectured that there is a bound on the dimension in terms of the size of the tile:

\begin{conj}[Gruslys, Leader, Tan \cite{gltan16}]
	For any positive integer $t$, there exists a number $d$ such that any tile $T \subset \mathbb{Z}$ with $|T| \leq t$ tiles $\mathbb{Z}^d$.
\end{conj}

This conjecture remains unresolved. The authors of \cite{gltan16} showed that if $d$ always exists then $d \to \infty$ as $t \to \infty$, by exhibiting a tile of size $3d-1$ that does not tile $\mathbb{Z}^d$. This gives a simple lower bound on $d$; better bounds would be of great interest.

\section*{Acknowledgements}

I would like to thank Vytautas Gruslys for suggesting this problem and for many helpful discussions, and Imre Leader for his encouragement and useful comments.

\vspace{5mm}
\noindent
Harry Metrebian\\
Trinity College\\
Cambridge\\
CB2 1TQ\\
United Kingdom

\vspace{2mm}
\noindent
rhkbm2@cam.ac.uk

\end{document}